\documentclass[11pt]{amsart}    
\usepackage{amsmath, amssymb, amsthm,pdfsync}
\usepackage{verbatim}
\usepackage{cite}
\usepackage{graphicx}
\usepackage{subfig, caption}
\usepackage{cancel}
\usepackage{enumitem}
\usepackage{wrapfig}
\usepackage[top=1.25in, bottom=1.25in, left=1in, right=1in]{geometry}
\usepackage{mathtools}
\usepackage{url}
\usepackage{bbm}

\setlist[itemize]{itemsep=-1mm}

\newtheorem{thm}{Theorem}
\newtheorem{lem}[thm]{Lemma}

\newtheorem{rem}[thm]{Remark}
\newtheorem{nota}[thm]{Notation}

\newcommand{\E}{\mathbb{E}}

\newcommand{\R}{\mathbb{R}}

\newcommand{\cbar}{\overline{c}}
\newcommand{\e}{e^{-y^2/8}}

\DeclareMathOperator{\Span}{Span}

\numberwithin{equation}{section}
\numberwithin{thm}{section}

\title{Population Stabilization in Branching Brownian Motion with absorption}
\author{Christopher Henderson}
\begin{document}

\maketitle

\begin{abstract}
We consider, through PDE methods, branching Brownian motion with drift and absorption.  It is well known that there exists a critical drift which separates those processes which die out almost surely and those which survive with positive probability.  In this work, we consider lower order corrections to the critical drift which ensures a non-negative, bounded expected number of particles and convergence of this expectation to a limiting non-negative number, which is positive for some initial data.   In particular, we show that the average number of particles stabilizes at the convergence rate $O(\log(t)/t)$ if and only if the multiplicative factor of the $O(t^{-1/2})$ correction term is $3\sqrt{\pi} t^{-1/2}$.  Otherwise, the convergence rate is $O(1/\sqrt{t})$.  We point out some connections between this work and recent work investigating the expansion of the front location for the initial value problem in Fisher-KPP~\cite{Bramson78, Bramson83, HNRR13, HNRR12}.
\end{abstract}

\section{Introduction}

In this paper we consider a probabilistic model, branching Brownian motion, through analytic means.  This process evolves as Brownian motion and splits at a constant rate into two particles undergoing independent Brownian motion.  In addition, we are interested in this model where the particles are pushed by a given drift, $-\dot{X}$ and where particles are killed upon reaching the origin.

This connection to the Fisher-KPP equation and its use as a model for populations undergoing selection have made branching Brownian motion the subject of intense interest in recent years~\cite{ABBS, ArguinBovierKistler, BrunetDerrida, BBS11, BBS13, HNRR12, HNRR13, Harris, HarrisHarris, HarrisHarrisKyprianou, Kesten, FangZeitouni, Maillard, MaillardZeitouni, Roberts}.  Early work by McKean, \cite{McKean}, focused on understanding the statistics of the rightmost particle of branching Brownian motion with neither drift nor absorption in order to understand solutions to the Fisher-KPP equation.   We describe in more detail the connection between this work and Fisher-KPP at the conclusion of this section.  In 1978, Kesten introduced drift and absorption at the origin into the model~in \cite{Kesten}.  In this work, he showed that $\dot{X} = 2$ is the critical drift, separating systems which die out with probability one from systems with positive survival probability.  Recently, more precise results have been obtained regarding the distribution of particles at or near the critical drift and the convergence of the statistics of the rightmost particle; see, for example, \cite{BrunetDerrida, ABBS, Harris, HarrisHarris, HarrisHarrisKyprianou, BBS11, BBS13, Maillard}.    The preceding is an incomplete bibliography, and the interested reader should investigate the references within the works referenced above.

Our perspective is slightly different.  We wish to understand the effect of lower order terms in expansion of the drift on the average number of surviving particles.  To be more explicit, let $N_t$ be the number of particles at time $t$ with the position of the $i$th particle given by $Y_t^i$, and define
\[v(t,x) = \E^x\left[ \sum_{i=1}^{N_t} v_0\left(Y_t^i\right)\right],\]
for some compactly supported function $v_0$.  Then $v$ solves the equation
\begin{equation}\label{e_linearized_kpp}
\begin{cases}
	v_t = v_{xx} + v,\\
	v(t,X(t)) = 0,\\
	v(0,x) = v_0(x),
\end{cases}
\end{equation}
see, e.g. the ``many-to-one'' lemma~\cite{HarrisHarrisKyprianou, HardyHarris}.  In particular, if $v_0$ is the indicator function of $[a,b]$, this is simply the expected number of particles in $[a,b]$.  In~\cite{HNRR13,HNRR12}, the authors, using \eqref{e_linearized_kpp} as an approximation for Fisher-KPP, show that, letting
\[
	\dot{X} = 2 + \frac{r}{t},
\]
the only choice for $r$ which yields non-trivial long-time behavior is $-3/2$.  We obtain a correction of order $t^{-3/2}$ which gives us more precise information on the average number of particles by a refinement of their methods and using a connection between the mass of the solution of \eqref{e_linearized_kpp}, $v$, with the normal derivative at the origin, $v_x(0)$.  This allows us to find an expansion for $\dot{X}$, independent of initial data, which yields faster convergence to the limiting mass.  We now state this precisely.

\subsection*{Statement of Results}

Before we state the main theorem, we give a bit of notation and recast our problem.  For ease of exposition, we use the following notation in order to omit tracking multiplicative constants that arise.
\begin{nota}
For two values $a$ and $b$, which may depend on time and various other data, we write
\[
	a \lesssim b,
\]
if there is some multiplicative constant, $C>0$, that is independent of time such that
\[
	a \leq Cb.
\]
Such a constant may depend on initial data or various other constants.
\end{nota}
\noindent We note that, despite this notation, we occasionally need to introduce an arbitrary constant into our equations.  Whenever we do, we denote by $C$ such an arbitrary constant which may change line by line, but which is independent of time.

In addition, in order to avoid the complications inherent in a moving boundary, we shift to a moving frame to obtain the equation
\begin{equation}\begin{cases} \label{e_shifted_linearized_kpp}
	v_t - \dot{X}(t)v_x  = v_{xx} + v,\\
	v(t,0) = 0,\\
	v(0,x) = v_0(x),
.\end{cases}\end{equation}

\begin{thm}\label{p_theorem}
Let $v$ satisfy \eqref{e_shifted_linearized_kpp} with $X(t)$ given by
\begin{equation}\label{e_front_location}
	X(t) = 2(t+1) - \frac{3}{2}\log(t+1) - \frac{\cbar}{\sqrt{t+1}}
.\end{equation}
Then there exists $\alpha_0\geq 0$ such that for any $p\in [1,\infty]$ we have
\begin{equation}\label{e_lp_convergence}
\lim_{t\to\infty} \|v(t) - \alpha_0 x e^{-x}\|_p = 0
.\end{equation}
In addition, we have that if $\cbar = 3\sqrt{\pi}$, then
\begin{equation}\label{e_fast_convergence}
	\left|\alpha_0 - \int_{0}^\infty v(t,x) dx\right| \lesssim \frac{\log(t)}{t}.
\end{equation}
Otherwise, we have that
\begin{equation}\label{e_slow_convergence}
	\frac{1}{\sqrt{t}} \lesssim \left|\alpha_0 - \int_{0}^\infty v(t,x) dx\right|
		\lesssim \frac{1}{\sqrt{t}}.
\end{equation}
Finally, for certain initial data, we have $\alpha_0>0$.
\end{thm}

The main focus of this paper is in obtaining the mass stabilization rates in \eqref{e_fast_convergence} and \eqref{e_slow_convergence}.  We state the convergence in $L^p$ and the positivity of $\alpha_0$ in order to reassure the skeptical reader that these rates have meaning.  The $L^p$ convergence and the positivity of $\alpha_0$ arises naturally in our work, and we make note when they become apparent.

Our proof of~\eqref{e_fast_convergence} in Theorem~\ref{p_theorem} involves no direct estimates on the mass of the solution.  In systems such as~\eqref{e_shifted_linearized_kpp}, there is a close relationship between the derivative of the solution at $x=0$ and its total mass.  Hence, we focus on obtaining estimates on the derivative of the solution at the origin.  To this end, the main ingredient of the proof of Theorem \ref{p_theorem} is the following lemma.
\begin{lem}\label{p_main_lemma}
Let $v$ satisfy \eqref{e_shifted_linearized_kpp} with $X(t)$ given by~\eqref{e_front_location}.  If $\cbar = 3\sqrt{\pi}$ we have
\[|v_x(0) - \alpha_0| \lesssim \frac{\log(t)}{t}.\]
On the other hand, if $\cbar \neq 3\sqrt{\pi}$, we have that, for $t$ sufficiently large
\[\frac{1}{\sqrt{t}} \lesssim |v_x(0) - \alpha_0| \lesssim \frac{1}{\sqrt{t}}.\]
\end{lem}

\subsection*{Connection With Front Speeds in Fisher-KPP}
Much of the renewed interest in understanding the system \eqref{e_linearized_kpp} and its probabilistic counterpart, branching Brownian motion with drift and absorption, lies in its connection to the Fisher-KPP equation
\begin{equation}\label{e_kpp}\begin{cases}
&u_t = u_{xx} + u(1-u),\\
&u(0,x) = u_0(x),
\end{cases}\end{equation}
where $u_0$ is some localized smooth function on $\R$ taking values in $[0,1]$.  This equation usually arises as a model for population dynamics, see e.g. \cite{Fife, Fisher, Murray, Xin09}.

The equation \eqref{e_kpp} was originally studied in the early twentieth century in \cite{Fisher, KPP}, and it was observed that traveling wave solutions to the system exist.  To be more explicit, there are global in time solutions of the form $u(t,x) = \phi_c(x-ct)$ for profiles $\phi_c$ and speeds $c\geq 2$.
Later, in \cite{AronsonWeinberger}, Aronson and Weinberger showed that solutions with more general initial data spread at the speed of the traveling wave and admit the same behavior, in that the steady state $u\equiv 1$ invades the unsteady state $u \equiv 0$.  Namely, given a solution to \eqref{e_kpp} where $u_0$ is compactly supported, non-negative, and non-zero, we have
\[
	\min_{|x| \leq ct} u(x) = 1, ~~\text{ for all } c < 2,
\]
and
\[
	\max_{|x| \geq ct} u(x) = 0, ~~\text{ for all } c > 2.
\]
In this case, as in Kesten's paper \cite{Kesten}, the critical speed is $2$.

In the celebrated papers \cite{Bramson78, Bramson83}, Bramson obtained more precise asymptotics of the front location with probabilistic methods.  More specifically, he showed that for any $m\in (0,1)$, there exists a shift $x_m$, depending on the initial conditions and $m$, such that
\[
	\{x \in \R: u(t,x) = m\} \subset \left[ 2t - \frac{3}{2}\log(t) - x_m - o(1), 2t - \frac{3}{2}\log(t) - x_m + o(1)\right].
\]
Recently, Roberts simplified the proofs of these results~\cite{Roberts}.  Ebert and van Saarloos, in \cite{EbertVanSaarloos}, obtained, through non-rigorous methods, the next term in the expansion.  Namely, using matched asymptotics, they argue that the front speed is given by
\begin{equation}\label{e_front_speed}
	2t - \frac{3}{2}\log(t) - x_m - \frac{3\sqrt{\pi}}{\sqrt{t}} + O(t^{-1}).
\end{equation}
Interestingly, though the constant term has dependence on the initial data, the lower order term is universal.  We point out that the expansion obtained by Ebert and van Saarloos is the same as the expansion we obtain in~\eqref{e_front_location} through Theorem \ref{p_theorem}.  In fact, part of the motivation of our work has been to provide some understanding of the $3\sqrt{\pi}$ term, as the Ebert and van Saarloos paper does not provide any interpretation beyond the matched asymptotics in the formal derivation.  The most recent works regarding the front location in Fisher-KPP  are~\cite{HNRR13, HNRR12} by Hamel et al.  In these papers the authors used PDE methods, which we have borrowed and expanded on here, in order to obtain results similar to those of Bramson at the expense of the precision in the constant term.  In addition, we mention a work in preparation,~\cite{BerestyckiBrunetHarris}, in which the authors, through probabilistic methods, investigate the value of $\cbar$ given in~\eqref{e_front_speed} in the Fisher-KPP context.

That our expansion for the critical drift in \eqref{e_linearized_kpp} is the same as Ebert and van Saarloos's expansion for the front location for \eqref{e_kpp} is not a surprise.  Solutions to \eqref{e_linearized_kpp} provide a convenient family of sub and super solutions to \eqref{e_kpp}, which are exceptionally faithful approximations of the tail of solutions to \eqref{e_kpp} provided that $X(t)$ is chosen carefully.  As we mentioned above, the authors in \cite{HNRR13, HNRR12, NRR}, use equations such as \eqref{e_linearized_kpp} in order to obtain precise results about the front position in Fisher-KPP.  In addition, we remind the reader that solutions of \eqref{e_linearized_kpp} and \eqref{e_kpp} are connected through their interpretation as statistics of branching Brownian motion \cite{McKean}.  The close relationship between these two probabilistic models has been leveraged to transfer understanding from one system to the other, see e.g. \cite{FangZeitouni, HNRR12, MaillardZeitouni} and many of the papers mentioned above.


\subsection*{Outline of the Paper}

The paper is organized as follows.  In Section \ref{s_theorem_proof}, we show how the mass stabilization rates given in Theorem \ref{p_theorem} follow from Lemma \ref{p_main_lemma}.  This boils down to leveraging a connection between the total mass of a solution to \eqref{e_shifted_linearized_kpp} and its derivative as $x = 0$ to obtain an ODE governing the total mass.  Then we apply the estimates provided by Lemma \ref{p_main_lemma} to conclude.

In Section \ref{s_self_similar}, we switch to self-similar variables in order to reduce the equation to a simple parabolic PDE with a decaying forcing term.  This change of variables also used in~\cite{HNRR13,HNRR12} and is convenient because it changes \eqref{e_shifted_linearized_kpp} into a PDE that has a discrete spectrum which we know explicitly.  This allows us to decompose our solution into three parts: a steady state, a slowly decaying function with vanishing normal derivative, and a rapidly decaying remainder.  This reduces Lemma \ref{p_main_lemma} to proving this decomposition.  We note that the convergence of $v$ to $\alpha_0 x e^{-x}$ is an easy consequence of this, and, as such, we omit the proof of it.

Finally, in Section \ref{s_lemma_proof}, we give the proof of this decomposition by explicitly solving for the steady state and the slowly decaying function and by obtaining estimates for the quickly decaying remainder. 
At some point in the analysis, the positivity of $\alpha_0$ becomes obvious; we share when this happens.

\smallskip\smallskip
\noindent\textbf{Acknowledgements: } 
The author would like to thank Lenya Ryzhik for suggesting the project and \'Eric Brunet for his lectures at the Banff workshop on deterministic and stochastic front propagation in 2010.

\section{Deducing Mass Stabilization Rates From Lemma \ref{p_main_lemma}}\label{s_theorem_proof}
\begin{proof}[Proof of \eqref{e_fast_convergence} and \eqref{e_slow_convergence}:]
We first cover the case where $\cbar = 3\sqrt{\pi}$.  Let $\alpha_0$ be the constant from Lemma \ref{p_main_lemma}.  Integrating equation \eqref{e_shifted_linearized_kpp} and defining
\[M(t) = \int_0^\infty v(t,x)dx - \alpha_0,\]
we obtain
\[
M'(t) - M = - v_x(0) + \alpha_0
.\]
Define $E(t) = e^{-t}M(t)$, and notice that $E'(t) = (M' - M)e^{-t}$.  Hence we have that
\[ E'(t) = (\alpha_0 - v_x(0))e^{-t}.\]
Integrating both sides of this and applying Lemma \ref{p_main_lemma}, we get
\[\begin{split}
	|M(t)|e^{-t}
		&= \left| \int_t^\infty E'(s)ds\right|
		= \left| \int_t^\infty (\alpha_0 - v_x(0))e^{-s}ds\right|\\
		&\lesssim \int_t^\infty \frac{\log(s)}{s}  e^{-s} ds
		\leq \frac{\log(t)}{t} \int_t^\infty e^{-s}ds
		= \frac{\log(t)}{t} e^{-t}
.\end{split}\]
Multiplying both sides by the exponential finishes the upper bound for this choice of $\cbar$.

If $\cbar \neq 3\sqrt\pi$, we may apply Lemma \ref{p_main_lemma} in the same manner to get the desired upper bound in the statement of the proof.  In order to obtain the lower bound, let $M$ and $E$ be as above.  For $t$ sufficiently large, we may apply Lemma \ref{p_main_lemma} to obtain that
\[\begin{split}
	|M(t)| e^{-t}
		&= \left| \int_t^\infty E'(s) ds\right|
		\gtrsim \int_t^\infty \frac{e^{-s}}{\sqrt{s}}ds\\
		&\geq \frac{1}{\sqrt{2t}} \int_t^{2t} e^{-s}ds
		\gtrsim \frac{e^{-t} - e^{-2t}}{\sqrt{t}}.
\end{split}\]
Multiplying both sides by $e^{t}$ and taking $t$ sufficiently large, we obtain the desired lower bound.  This finishes the proof.
\end{proof}

\section{Self-Similar Variables}\label{s_self_similar}

In order to understand the solutions to \eqref{e_shifted_linearized_kpp}, we change variables a number of times following the development in \cite{HNRR13,HNRR12}.  First, we remove an exponential to obtain the function $\overline{v} = e^x v$ satisfying
\begin{equation}
\overline{v}_t + \left(\frac{3}{2t} - \frac{\cbar}{2t^{3/2}}\right)\overline{v}_x = \overline{v}_{xx} + \left(\frac{3}{2t} - \frac{\cbar}{2t^{3/2}}\right)\overline{v}.
\end{equation}
Changing to self-similar variables $\tau = \log(1+t)$ and $y = x(1+t)^{-1/2}$, we obtain
\begin{equation}
w_\tau - \frac{y}{2} w_y - w_{yy} - \frac{3}{2}w =  \left(\frac{\cbar}{2e^{\tau}}-\frac{3}{2e^{\tau/2}}\right) w_y  - \frac{\cbar}{2e^{\tau/2}} w.
\end{equation}
Let $W(\tau,y) = e^{\tau/2} e^{y^2/8} w(\tau,y)$ and we get
\begin{equation}\label{e_pde_changed_var}
W_\tau + MW = \left(\frac{\cbar}{2e^{\tau}} - \frac{3}{2e^{\tau/2}}\right)\left(W_y - \frac{y}{4}W \right) - \frac{\cbar}{2e^{\tau/2}} W,
\end{equation}
where
\[M = -\partial_y^2 + \left(\frac{y^2}{16} - \frac{3}{4}\right).\]
 We work mainly with $W$ in the sequel.  Notice that the operator on the left hand side of~\eqref{e_pde_changed_var} is the equation of a simple harmonic oscillator with a decaying forcing term.  We use the fact that we understand the eigenvalues and eigenfunctions of this operator in the analysis that follows.

In these coordinates, Lemma \ref{p_main_lemma} can be reduced to proving the following decomposition.
\begin{lem}\label{p_decomposition}
Let $W$ satisfy \eqref{e_pde_changed_var} with smooth, compactly supported initial data.  Then there exists $\alpha$, $g$, and $R$ such that
\begin{equation}\label{e_decomposition}
W(\tau,y)
	=\alpha y e^{-y^2/8} + e^{-\tau/2} g(y) + R(\tau, y)
.\end{equation}
In addition, $|R_y(0)| \lesssim \tau e^{-\tau}$. Finally, $g_y(0) = 0$ if and only if $\cbar = 3\sqrt{\pi}$
\end{lem}

We prove this lemma in the sequel by decomposing $W$ into three functions: one part is the steady solution of the equation, one is a slowly decaying function with zero derivative at $z = 0$, and one is a quickly decaying function.  First, we show that $W$ is bounded in $L^2$ and converges to the steady state at the rate $e^{-\tau/2}$.  This also shows the existence of $\alpha$ above.  Then we prove the existence of $g$ and use the fact that we may solve for it explicitly.  Finally, we leverage these facts to prove the existence of $R$.

%

\section{Proof of Lemma \ref{p_decomposition}}\label{s_lemma_proof}

Before we begin the proof, notice that $M$ defines a non-negative definite, symmetric quadratic form on the space
\begin{equation}\label{e_X}
X := H^1\cap\{\phi \in L^2: y\phi \in L^2\}
,\end{equation}
which we call $Q$.  Namely, for all $\phi \in X$, we define
\begin{equation}\label{e_Q}
Q(\phi) := \int \phi (M\phi) dy.
\end{equation}
Note that $Q$ satisfies the following inequality
\begin{equation}\label{e_q_split}
\int \frac{y}{4} \phi^2 dy \leq Q(\phi) + \|\phi\|_2^2
.\end{equation}
We use this inequality often in the sequel.

Let $e_0, e_1, \dots$ denote the eigenfunctions of $M$ and we know by \cite{HNRR13} that the first two eigenvalues are $0$ and $1$.  Hence $Q$ is non-negative on $X$ and $Q(\phi) \geq \|\phi\|_2^2$ on $\Span\{e_1, e_2, \dots\}$.  Moreover, we know that
\[e_0(y) = \frac{1}{\sqrt{2\sqrt{\pi}}} y\e.\]

First, we show that $W$ converges in $L^2$ to the steady state at the rate $e^{-\tau/2}$.  In addition, this shows the existence of $\alpha$ in Lemma \ref{p_decomposition} since $e_0$ is the steady state of \eqref{e_pde_changed_var}.  We remark that the potential positivity of the total mass in Theorem \ref{p_theorem} follows from the work below.
\begin{lem}\label{p_decomposition_0th_order}
Suppose that $W$ satisfies \eqref{e_pde_changed_var}.  Then there exists $\alpha$ such that
\[
	\|W - \alpha y e^{-y^2/8}\|_2 \lesssim e^{-\tau/2}
.\]
\end{lem}
\begin{proof}
First, we show that $W$ is bounded in $L^2$.  To this end, multiplying \eqref{e_pde_changed_var} by $W$ and integrating by parts we get
\[ \frac{1}{2} \frac{d}{dt} \int |W|^2 dy + Q(W)
	= \int \left[\left(\frac{\cbar}{2e^\tau} - \frac{3}{2e^{\tau/2}}\right) y W^2 - \frac{\cbar}{2e^{\tau/2}} W^2\right] dy
	\lesssim e^{-\tau/2} \left[ Q(W) + \|W\|_2^2\right]
,\]
where we used \eqref{e_q_split} to obtain the second inequality.  Using the non-negativity of $Q$ and solving the differential inequality, we obtain
\[\frac{d}{dt} \|W\|_2^2  \lesssim e^{-\tau/2} \|W\|_2^2.\]
Integrating this gives us that $W$ is uniformly bounded in $L^2$.

Now we finish the proof in two steps. First, we look at the projection of $W$ onto $e_0$.  This gives us the steady state.  Then we look at the component of $W$ orthogonal to $e_0$.  To be explicit, we decompose $W$ as
\begin{equation}\label{e_w_decomposition}
	W = W_1(\tau) e_0 + \tilde W
,\end{equation}
where $\tilde W$ is an element of $\Span\{e_1,e_2,\dots\}$.

In order to understand $W_1$, multiply \eqref{e_pde_changed_var} by $e_0$ and integrate by parts to obtain
\[
\left|(W_1)_\tau\right|
	\lesssim e^{-\tau/2} \|W\|_2
	\lesssim e^{-\tau/2}
.\]
Hence there exists $\alpha'$ such that $W_1$ tends to $\alpha'$ as $\tau$ tends to infinity.  Moreover, we have that $|\alpha' - W_1| \lesssim e^{-\tau/2}$.  Hence we need only show that $\tilde W$ decays fast enough in order to finish the proof.

To obtain the decay of $\tilde W$, we use \eqref{e_w_decomposition} in \eqref{e_pde_changed_var} to obtain
\[\begin{split}
\tilde W_\tau + M\tilde W 
	&= e^{-\tau/2} \left[ \left(\frac{\cbar}{2e^{\tau/2}} - \frac{3}{2}\right)\left(\tilde W_y - \frac{y}{4}\tilde W \right) - \frac{\cbar}{2} \tilde W  \right]\\
	&~~~~+ e^{-\tau/2} \left[\left(\frac{\cbar}{2e^{\tau}} - \frac{3}{2e^{\tau/2}}\right)\left(W_1 (e_0)_y - \frac{y}{4}W_1 e_0 \right) - \frac{\cbar}{2e^{\tau/2}} W_1 e_0 \right] - (W_1)_\tau e_0
.\end{split}\]
Noting that $\tilde W$ lives in the span of $e_1, e_2, \dots$, we have that $Q(\tilde W) \geq \|\tilde W\|_2^2$.  Hence, when we multiply the equation above by $\tilde W$, integrate by parts and use our inequality on $Q$, we obtain
\[\frac{1}{2} \frac{d}{dt} \|\tilde W\|_2^2 + \left(1 - C e^{-\tau/2}\right)\|\tilde W\|_2^2
	\lesssim e^{-\tau/2} \|\tilde W\|_2.\]
Solving this differential inequality yields
\[\|\tilde W\|_2 \lesssim e^{-\tau/2},\]
finishing the proof.
\end{proof}

\begin{rem}
By changing coordinates, we see that
\[
	W_0(y) = e^{y^2/8} e^{y} v_0(y),
\]
which gives us that
\[
	\langle W_0, y e^{-y^2/8}\rangle
		=  \int_0^\infty y e^{y} v_0(y) dy
		=  \int_0^\infty \xi v_0(\xi) d\xi.
\]
Hence, if $\int \xi v_0(\xi)d\xi$ is large enough, then $\alpha_0$ must be positive.  When our equation, \eqref{e_shifted_linearized_kpp}, is used to approximate Fisher-KPP, this may be overcome by either choosing a larger initial condition or running the system for sufficiently long in order that this integral is large, depending on whether one is looking for a supersolution or subsolution.
\end{rem}

Now we investigate $g$ and $R$ in \eqref{e_decomposition}.  The ansatz implicit in \eqref{e_decomposition} gives us the following equation
\[\begin{split}
- \frac{e^{-\tau/2}}{2}g+&e^{-\tau/2}Mg + R_\tau + MR\\
	&=  e^{-\tau/2} \left[\frac{3\alpha}{4} y^2 \e - \frac{\alpha\cbar}{2}y\e - \frac{3\alpha}{2} \e\right]\\
	&~~~~+  e^{-\tau}\left[\frac{\alpha\cbar}{2}\e - \frac{\alpha\cbar}{4} y^2 \e - \frac{3}{2} g_y + \frac{3y}{8} g - \frac{\cbar}{2}g + \frac{\cbar e^{-\tau/2}}{2}\left(g_y - \frac{y}{4}g\right)   \right]\\
	&~~~~+  e^{-\tau/2}\left[ \left(\frac{ \cbar e^{-\tau/2}}{2} - \frac{3}{2}\right)\left(R_y - \frac{y}{4}R\right) - \frac{\cbar}{2}R \right]
,\end{split}\]
where $\alpha$ is as in Lemma \ref{p_decomposition_0th_order}.  We separate this into equations for $g$ and $R$ by associating the terms on the right of order $e^{-\tau/2}$ with $g$ and the rest with $R$.  This yields
\begin{equation}\label{e_g}
Mg - \frac{g}{2} =  \alpha\e \left[\frac{3}{4} y^2  - \frac{\cbar}{2}y - \frac{3}{2} \right]
,\end{equation}
and
\begin{equation}\label{e_R}
R_\tau + MR
	= e^{-\tau} f(y,\tau) +  e^{-\tau/2}\left[ \left(\frac{ \cbar e^{-\tau/2}}{2} - \frac{3}{2}\right)\left(R_y - \frac{y}{4}R\right) - \frac{\cbar}{2}R \right]
,\end{equation}
where $f$ is given by
\[
f(\tau,y) 
	= \left[ \frac{\alpha\cbar}{2}\e - \frac{\alpha\cbar}{4} y^2 \e - \frac{3}{2} g_y + \frac{3y}{8} g - \frac{\cbar}{2}g + \frac{\cbar e^{-\tau/2}}{2}\left(g_y - \frac{y}{4}g\right)\right]
.\]
We first show that \eqref{e_g} is well defined and that $g$ satisfies the properties claimed.  Then we show \eqref{e_R} is well-defined as well.

\begin{lem}\label{p_decomposition_g}
There exists a smooth solution $g \in X$ of \eqref{e_g} which is locally bounded in $C^1$.  In addition, $g_y(0) = 0$ if and only if $\cbar = 3\sqrt{\pi}$.  
\end{lem}
\begin{proof}
First we show, abstractly, that such a solution exists.  Then we write down an explicit solution to the equation.  This explicit solution allow us to understand $g_y(0)$.

In order to show the existence of a solution $g\in X$, we proceed as in Lemma \ref{p_decomposition_0th_order}.  Namely, we write
\begin{equation}
	g = g_1e_0 + \tilde g
,\end{equation}
where $\tilde g$ is in the span of $e_1, e_2, \dots$. To bound $g_1$, we multiply \eqref{e_g} by $e_0$ and integrate.  This gives an explicit formula for $g_1$ independent of $\tilde g$.  Namely
\begin{equation}\label{e_g_1}
g_1 = -2 \alpha \int \left[\frac{3}{4}y^2 - \frac{\cbar}{2} - \frac{3}{2} \right] e_0(0) \e dy
.\end{equation}

Then, fixing $g_1$ as this value, we simply write down the equation for $\tilde g$ given by
\begin{equation}\label{e_tilde_g}
	M\tilde g - \frac{\tilde g}{2}
		= \alpha\left[\frac{3}{4}y^2 - \frac{\cbar}{2}y - \frac{3}{2}\right] + \frac{g_1}{2} e_0.
\end{equation}
On $\Span\{e_1, e_2, \dots\}$ with the norm of $X$, the operator $M - 1/2$ is coercive.  Here we are using that for $\phi \in X\cap\Span\{e_1,e_2,\dots\}$, we have $Q(\phi) \geq \|\phi\|_2^2$.  Hence, the Lax-Milgram theorem implies that \eqref{e_tilde_g} is uniquely solvable in $\Span\{e_1,e_2,\dots\}$.  This, along with \eqref{e_g_1}, gives us that \eqref{e_g} is uniquely solvable in $X$.  The standard elliptic theory, as in \cite{GilbargTrudinger}, then, implies that $g$ is in fact in $H_{loc}^k$ for every $k$.  This, in addition, implies that $g$ is smooth and locally bounded in $C^1$.

We now solve \eqref{e_g} explicitly.  By changing variables to $z = y^2/4$ and letting $G = e^{z/2}g$, we obtain the equation:
\[
	-zG_{zz} - \left(\frac{1}{2} - z\right)G_z - G =  \alpha \left[3z  - \cbar\sqrt{z} - \frac{3}{2} \right]
.\]
Following the work of Ebert and van Saarloos in \cite{EbertVanSaarloos}, shows us that the explicit solution is of the form
\[G(z) = \alpha \left[\frac{3}{2} + 2\cbar\sqrt{z} - \frac{3}{2} F_2(z) + a_1(1 - 2z) + a_2 H(z)  \right]\]
where $F_2$ and $H$ are given by 
\begin{equation}\label{e_hypergeometric}
F_2(z) = \sqrt{\pi} \sum_{n=2}^\infty \frac{z^n}{n(n-1) \Gamma(1/2 + n)},
	~~\text{ and }~~
	H(z) = -\frac{\sqrt{z}}{4} \sum_{n=0}^\infty \frac{z^n}{n!} \frac{\Gamma(-1/2 + n)}{\Gamma(3/2 + n)}
.\end{equation}
As in \cite{EbertVanSaarloos}, one may check that
\[
	\lim_{z\to\infty} \frac{H(z)}{z^{-3/2}e^z} = -\frac{1}{4}, ~~\text{ and }~~
		\lim_{z\to\infty} \frac{F_2(z)}{z^{-3/2}e^z} = \sqrt{\pi}
.\]
Hence $H$ and $F_2$ are clearly not in $L^2$ in our original variables, even with the additional $e^{-z/2}$ factor.  Thus, we must choose $a_2$ such that these terms cancel at $z = \infty$.  In addition, we must choose $a_1$ such that $G(0) = 0$.  Hence, we obtain
\[\begin{split}
	G(z) &= \alpha \left[\frac{3}{2} + 2\cbar\sqrt{z} - \frac{3}{2} G_2(z) + \frac{-3}{2}(1 - 2z) - 6\sqrt{z}H(z)  \right]\\
	&= \alpha \left[2\cbar\sqrt{z} - \frac{3}{2} G_2(z) + 3z - 6 \sqrt{\pi} H(z)  \right]\\
.\end{split}\]
By uniqueness, we may return to the original variables to obtain an explicit formula for $g_y(0)$.  With this formula, we can easily see that $g_y(0) = 0$ if and only if $2\cbar = 6\sqrt{\pi}$, finishing the proof.
\end{proof}

Lemmas \ref{p_decomposition_0th_order} and \eqref{p_decomposition_g} tell us that $R$ must decay to zero in $L^2$ as $\tau$ tends to infinity.  This is a key fact that we use in the following lemma, which finishes the proof of Lemma \ref{p_decomposition}.
\begin{lem}\label{p_decomposition_R}
Let $R$ satisfy \eqref{e_R} with $\alpha$ and $g$ given above.  Then we have the following bounds
\[\|R\|_2
	\lesssim \tau e^{-\tau}, ~~\text{ and }~~ |R_y(0)| \lesssim \tau e^{-\tau}
.\]
\end{lem}
\begin{proof}
We proceed by decomposing $R$ as we did $W$ and $g$ in the proofs of Lemmas \ref{p_decomposition_0th_order} and \ref{p_decomposition_g}.  Namely, let
\begin{equation}\label{e_R_decomposition}
	R = R_1(\tau) e_0 + \tilde R
,\end{equation}
where $\tilde R$ is orthogonal to $e_0$.  To obtain a bound on $R_1$, we first note that Lemmas \ref{p_decomposition_0th_order} and \ref{p_decomposition_g} imply that $R$ decays to zero in $L^2$ as least at the rate $e^{-\tau/2}$.  Then we multiply \eqref{e_R} by $e_0$ and integrate to obtain
\[
(R_1)_\tau
	= e^{-\tau} \langle f, e_0\rangle -  e^{-\tau/2}\langle R, \left(\frac{ \cbar e^{-\tau/2}}{2} - \frac{3}{2}\right)\left((e_0)_y + \frac{y}{4}e_0\right) + \frac{\cbar}{2}e_0 \rangle
.\]
Applying Cauchy-Schwarz, we obtain
\begin{equation}\label{e_prelim_remainder1}
	|(R_1)_\tau| \lesssim e^{-\tau} + e^{-\tau/2} \|R\|
.\end{equation}
Since we know that $\|R\|_2\lesssim e^{-\tau/2}$, it follows that
\[
	|(R_1)_\tau| \lesssim e^{-\tau}
.\]
This gives us that
\[
	|R_1(\tau)|
		=|R_1(\infty) - R_1(\tau)|
		=\left| \int_\tau^\infty (R_1)_\tau(s) ds\right|
		\lesssim \int_\tau^\infty e^{-s} ds
		= e^{-\tau}
.\]
This is the desired bound on $R_1$.

In order to finish the proof, we need only bound $\tilde R$ in $L^2$. To this end we use the decomposition \eqref{e_R_decomposition} along with \eqref{e_R}, to note that $\tilde R$ satisfies
\[\begin{split}
\tilde R_\tau + M\tilde R
	&= e^{-\tau} f(y,\tau) +  e^{-\tau/2}\left[ \left(\frac{ \cbar e^{-\tau/2}}{2} - \frac{3}{2}\right)\left(\tilde R_y - \frac{y}{4}\tilde R\right) - \frac{\cbar}{2}\tilde R \right]\\
	&~~~~+  R_1 e^{-\tau/2}\left[ \left(\frac{ \cbar e^{-\tau/2}}{2} - \frac{3}{2}\right)\left((e_0)_y - \frac{y}{4}e_0\right) - \frac{\cbar}{2}e_0 \right] - (R_1)_\tau e_0
.\end{split}\]
Multiplying this by $\tilde R$ and integrating by parts yields
\[
\frac{1}{2} \frac{d}{dt}\|\tilde R\|_2^2 + Q(\tilde R)
	\lesssim  e^{-\tau}\|\tilde R\|_2 +  e^{-\tau/2}\left( \int (y + 1)\tilde R^2 dy\right) +  e^{-\tau/2} R_1 \|\tilde R\|_2 + |(R_1)_\tau| \|\tilde R\|_2
.\]
Again using the inequality \eqref{e_q_split} along with the inequality on $R_1$ and $(R_1)_\tau$ that we just obtained, we note that
\begin{equation}\label{e_prelim_remainder2}
\frac{d}{dt}\|\tilde R\|_2^2 + (2 - C  e^{-\tau/2})\|\tilde R\|_2^2
	\lesssim e^{-\tau} \|\tilde R\|_2
,\end{equation}
where $C$ is some universal constant.  Solving this differential inequality gives us the desired inequality
\[\|\tilde R\|_2 \lesssim \tau e^{-\tau},\]
which finishes the proof.

To finish the proof we need to bound $R_y(0)$.  This, however, is a simple consequence of our bounds on the $L^2$ norm of $R$ and the right hand side of \eqref{e_R}.  Indeed, with these, the standard parabolic regularity theory, which may be found, for example, in \cite{KrylovHolder, KrylovSobolev}, give us the desired bound on $R_y(0)$, finishing the proof.
\end{proof}

\bibliography{vanSaar_const.bib}{}
\bibliographystyle{amsplain}

\end{document}